\documentclass[11pt]{article}

\usepackage[english]{babel}
\usepackage{amsmath,amsthm,amssymb,bm,thmtools,enumitem,mathrsfs,mathtools,titlesec,tikz,tikz-cd}
\usepackage[left=2.4cm,right=2.4cm,top=2.4cm,bottom=2.4cm,bindingoffset=0cm]{geometry}

\usepackage{quoting}
\quotingsetup{vskip=0pt}

\setlist[enumerate]{topsep=2pt,label=\textup{(\arabic*)},leftmargin=2em,labelsep=.5em,itemsep=0pt}

\usepackage[T1]{fontenc}

\declaretheoremstyle[
  spaceabove=\topsep, spacebelow=2ex,
  headfont=\normalfont\bfseries,
  notefont=\mdseries, notebraces={(}{)},
  bodyfont=\normalfont\itshape,
  postheadspace=.5em,
  qed=\qedsymbol
]{mystyle}

\declaretheoremstyle[
  spaceabove=\topsep, spacebelow=6pt,
  headfont=\normalfont\bfseries,
  notefont=\mdseries, notebraces={(}{)},
  bodyfont=\normalfont,
  postheadspace=.5em,
  qed=\qedsymbol
]{mydefstyle}

\swapnumbers
\theoremstyle{mystyle}
\declaretheorem[numberlike=subsection]{proposition}
\declaretheorem[numberlike=subsection]{theorem}
\declaretheorem[numberlike=subsection]{corollary}

\theoremstyle{mydefstyle}
\declaretheorem[numberlike=subsection]{definition}
\declaretheorem[numberlike=subsection]{remark}
\declaretheorem[numberlike=subsection]{remarks}
\declaretheorem[numberlike=subsection]{example}

\numberwithin{equation}{subsection}

\titleformat{\section}[block]
  {\filcenter\normalfont\large\bfseries}{\thesection.}{.5em}{}
\titleformat{\subsection}[runin]
  {\normalfont\bfseries}{\thesubsection.}{.5em}{}

\titlespacing*{\section}{0pt}{5ex plus .2ex minus 1ex}{3ex plus .2ex minus 1ex}
\titlespacing*{\subsection}{0pt}{2ex plus .2ex minus 1.5ex}{.5em}

\newcommand{\bbC}{\mathbb{C}}
\newcommand{\bbF}{\mathbb{F}}

\newcommand{\bbQ}{\mathbb{Q}}

\newcommand{\bbZ}{\mathbb{Z}}

\newcommand{\cA}{\mathscr{A}}

\newcommand{\cF}{\mathscr{F}}

\newcommand{\cH}{\mathscr{H}}

\newcommand{\cO}{\mathscr{O}}
\newcommand{\cP}{\mathscr{P}}

\newcommand{\cS}{\mathscr{S}}

\newcommand{\cX}{\mathscr{X}}

\newcommand{\overbar}[1]{\mkern 4mu\overline{\mkern-4mu#1}}

\newcommand{\CH}{\mathrm{CH}}

\newcommand{\uD}{\mathrm{D}}

\newcommand{\Ker}{\mathrm{Ker}}

\newcommand{\bdd}{\mathrm{b}}

\newcommand{\ch}{\mathrm{ch}}

\newcommand{\dR}{\mathrm{dR}}
\newcommand{\id}{\mathrm{id}}

\newcommand{\num}{\mathrm{num}}

\newcommand{\pr}{\mathrm{pr}}

\newcommand{\Td}{\mathrm{Td}}

\newcommand{\tto}{\longrightarrow}
\newcommand{\isomarrow}{\xrightarrow{\,\sim\, }}

\let\epsilon=\varepsilon
\let\phi=\varphi

\begin{document}

\begin{center}
\textbf{\LARGE A remark on the paper of Deninger and Murre}
\medskip

\textit{by}
\medskip

{\Large Ben Moonen}
\vspace{6mm}

In memory of Jaap Murre
\end{center}
\vspace{6mm}

{\small %% for abstract and MSC

\noindent
\begin{quoting}
\textbf{Abstract.} We show that the results proven by Deninger and Murre in their paper~\cite{DenMurre} directly imply that the Chern classes of the de Rham bundle of an abelian scheme are torsion elements in the Chow ring, a result that was later proven by van der Geer. We also discuss several results about the orders of these classes.
\medskip

\noindent
\textit{AMS 2020 Mathematics Subject Classification:\/} 14K05, 14C15
\end{quoting}

} %% end of small
\vspace{2mm}

%%%
%%%
\section{Introduction}

\subsection{}
In the study of Chow motives, abelian schemes form a rare class of examples where our understanding is rather complete and satisfactory. In the paper~\cite{DenMurre} by Deninger and Murre, earlier results of Mukai~\cite{MukaiFourier} and Beauville~\cite{BeauvFourier}, \cite{BeauvChow} were lifted to the level of Chow motives. These results were further refined and completed by K\"unnemann~\cite{Kunnemann},~\cite{KunnemannMotives}. 

One of the key points in the theory is, for $X \to S$ an abelian scheme, the introduction of a Fourier transform~$\cF_X$ and the proof that $\cF_{X^t} \circ \cF_X = (-1)^g \cdot [-1]_*$, where $X^t$ is the dual abelian scheme and $g = \dim(X/S)$. Over a field, this calculation was done by Mukai on the level of the derived category, and this readily implies the analogous result on Chow rings. For Chow rings of abelian schemes, the same argument, which relies on a Grothendieck--Riemann--Roch calculation, still works but the result that one arrives at involves some extra factors which are Todd classes of Hodge bundles. This is not, however, how Deninger and Murre argue. Rather, the desired identity is first proven by them modulo terms in higher degrees, and only at the very end, as a consequence of the existence of a Beauville decomposition, they establish the precise relation $\cF_{X^t} \circ \cF_X = (-1)^g \cdot [-1]_*$.

The remark that we want to make here, is that if we compare the two approaches, an interesting fact drops out. Namely, if $\cH = \cH^1_\dR(X/S)$ denotes the de Rham bundle of~$X$ over~$S$, which is a vector bundle on~$S$ of rank~$2g$, one obtains that $\Td(\cH)=1$. Moreover, if there exists a separable isogeny between~$X$ and~$X^t$ then it follows that $c_i(\cH) = 0$ in $\CH(S;\bbQ)$ for all $i\geq 1$, a result that was proven by van der Geer~\cite{vdG} several years later.

In the final section, we discuss some results that enable us to obtain bounds on the orders of the Chern classes of~$\cH$ in the (integral) Chow ring of~$S$. It was shown in~\cite{IAFD} that Fourier duality can be refined to a theory with coefficients in the ring $\Lambda = \bbZ[1/(2g+d+1)!]$, where $g = \dim(X/S)$ and $d = \dim(S)$. The relation $\Td(\cH) = 1$ is valid in $\CH(S;\Lambda)$, and for primes $\ell > 2g+d+1$ this can be used to bound the $\ell$-adic valuations of the orders of the classes $c_{2i}(\cH)$. (If we assume there exists a separable isogeny $X\to X^t$, the odd Chern classes of~$\cH$ vanish.) The bounds we obtain are closely related to bounds that, over~$\bbC$, have been obtained by Maillot and R\"ossler in~\cite{MailRoss}.

\subsection{Notation.} We write $B_i$ for the $i$th Bernoulli number. (Note that $B_i = 0$ if $i$ is odd.) If $X$ is a scheme of finite type over a Dedekind ring and $\Lambda$ is a commutative ring, we define $\CH(X;\Lambda) = \CH(X) \otimes \Lambda$, where $\CH(X)$ is the group of algebraic cycles on~$X$ modulo rational equivalence.

\subsection{}
This note is dedicated to the memory of Jaap Murre. His work and his lectures have been a source of inspiration for me, and I will remember him for his great kindness.

%%%
%%%
\section{A vanishing result for Todd classes of de Rham bundles}\label{sec:FD}

\subsection{}
Let $S$ be a connected base scheme that is smooth of finite type over a Dedekind ring. As pointed out by K\"unnemann in~\cite{KunnemannArChow}, Remark~1.1, the results proven by Deninger and Murre in~\cite{DenMurre} are valid for abelian schemes over~$S$. If $X$ is a smooth $S$-scheme of finite type, we denote by $\uD^\bdd(X)$ the bounded derived category of coherent $\cO_X$-modules, and by $\ch \colon \uD^\bdd(X) \to \CH(X;\bbQ)$ the Chern character.

\subsection{}
If $\pi \colon X \to S$ is an abelian scheme of relative dimension~$g$ with zero section~$e$, we define 
\[
E_{X/S} = e^*\Omega^1_{X/S}\, ,
\]
which is called the Hodge bundle of~$X$ over~$S$. We have $\Omega^1_{X/S} = \pi^*E_{X/S}$ and $E_{X/S} \cong \pi_*\Omega^1_{X/S}$. 

Let $\pi^t \colon X^t\to S$ be the dual abelian scheme, the zero section of which we call~$e^t$. For its Hodge bundle we have the relation $E_{X^t/S}^\vee = R^1\pi_*(\cO_X)$. In general, $E_{X^t/S}$ does not seem to have an obvious relation to~$E_{X/S}$. However, $\det(E_{X/S}) \cong \det(E_{X^t/S})$. (See \cite{GW}, Corollary~27.230). If there exists a separable isogeny $f \colon X \to X^t$ then $f$ induces an isomorphism $f^*\Omega^1_{X^t/S} \isomarrow \Omega^1_{X/S}$, so in this case $E_{X/S}$ and~$E_{X^t/S}$ are isomorphic. 

The de Rham bundle of $X$ over~$S$ is defined as
\[
\cH^1_\dR(X/S) = R^1\pi_*(\Omega^\bullet_{X/S})\, ,
\]
which is a vector bundle of rank~$2g$ over~$S$. By the degeneration of the Hodge--de Rham spectral sequence, it sits in a short exact sequence
\begin{equation}\label{eq:sesdR}
0 \tto E_{X/S} \tto \cH^1_\dR(X/S) \tto E_{X^t/S}^\vee \tto 0\, .
\end{equation}

\subsection{}
Let $\cP$ be the Poincar\'e bundle on $X\times_S X^t$, which is a line bundle with rigidification along $\bigl(X\times e^t(S)\bigr) \cup \bigl(e(S) \times X^t\bigr)$. Let $\cP^t$ be the Poincar\'e line bundle of~$X^t$. If $\sigma \colon X\times_S X^t \isomarrow X^t\times_S X$ is the isomorphism that reverses the factors, $\sigma^*\cP^t \cong \cP$ as rigidified line bundles on~$X\times_S X^t$.

On derived categories, the Poincar\'e bundle gives a Fourier transformation
\[
\cS = \cS_X \colon \uD^\bdd(X) \to \uD^\bdd(X^t)\quad \text{by}\quad \cS(K) = R\pr_{2,*}\bigl(\pr_1^*(K) \otimes^\mathrm{L} \cP\bigr)\, ,
\] 
where we write $\pr_1 \colon X\times_S X^t \to X$ and $\pr_2 \colon X\times_S X^t \to X^t$ for the projection morphisms. Similarly, $\cP^t$ gives rise to a Fourier transform~$\cS^t$ in the opposite direction. For $K \in \uD^\bdd(X)$ we have
\begin{equation}\label{eq:StS}
\cS^t \circ\cS (K) \cong [-1]_X^* \Bigl(K \otimes^\mathrm{L} \pi^*\det(E_{X/S})^{-1}[-g]\Bigr)\, .
\end{equation}
(See \cite{GW}, Theorem~27.243.)

\begin{remark}
A key fact that is used in the proof of \eqref{eq:StS} is that
\[
R^i\pr_{1,*}(\cP) = 
\begin{cases} 0 & \text{for $i\neq g$,}\\
e_*\bigl(\det(E_{X/S})^{-1}\bigr) & \text{for $i=g$.} \end{cases}
\]
(See \cite{GW}, Corollary~27.230.) This does not agree with what Deninger and Murre state in~\cite{DenMurre}, Proposition~2.8.1. The mistake is in their claim (p.~210, third displayed formula, translated into our notation) that $R^1\pi_*\cO_X \cong E_{X^t/S}$; correct is: $R^1\pi_*\cO_X \cong E_{X^t/S}^\vee$.
\end{remark}

\subsection{}
On Chow groups, the Fourier transform $\cF = \cF_X \colon \CH(X;\bbQ) \to \CH(X^t;\bbQ)$ is defined by
\[
\cF(\gamma) = \pr_{2,*}\bigl(\pr_1^*(\gamma)\cdot \ch(\cP)\bigr)\, .
\]
If we view $\CH(X;\bbQ)$ and $\CH(X^t;\bbQ)$ as algebras over~$\CH(S;\bbQ)$ (via~$\pi^*$, respectively~$\pi^{t,*}$) then $\cF$ is a homomorphism of~$\CH(S;\bbQ)$-algebras.

For $K\in \uD^\bdd(X)$ we have, by Grothendieck--Riemann--Roch,
\begin{multline*}
\ch\bigl(\cS(K)\bigr) = \ch\Bigl(R\pr_{2,*}\bigl(\pr_1^*(K) \otimes^L \cP\bigr)\Bigr) = \pr_{2,*}\bigl(\ch\bigl(\pr_1^*(K) \otimes^L \cP\bigr) \bigr) \cdot \pi^{t,*}\Td(E_{X/S}^\vee)\\
= \pr_{2,*}\bigl(\pr_1^*\ch(K) \cdot \ch(\cP)\bigr) \cdot \pi^{t,*}\Td(E_{X/S}^\vee) = \cF\bigl(\ch(K)\bigr) \cdot \pi^{t,*}\Td(E_{X/S}^\vee)\, ,
\end{multline*}
where we use that the relative tangent bundle of $\pr_2 \colon X\times_S X^t \to X^t$ is the pullback of $\pi^{t,*}(E_{X/S}^\vee)$. This means that we have a commutative diagram
\begin{equation}\label{eq:SvsFdiag}
\begin{tikzcd}[column sep=3.8cm]
\uD^\bdd(X) \ar[r,"\cS"] \ar[d,"\ch"] & \uD^\bdd(X^t) \ar[d,"\ch"]\\
\CH(X;\bbQ) \ar[r,"\alpha\mapsto \cF(\alpha) \cdot \pi^{t,*}\Td(E_{X/S}^\vee)"] & \CH(X^t;\bbQ)
\end{tikzcd}
\end{equation}

\begin{proposition}\label{prop:FtF}
For every $\alpha \in \CH(X;\bbQ)$ we have the relation
\begin{equation}\label{eq:FtF}
\cF^t \circ \cF(\alpha) \cdot \pi^*\Bigl(\Td(E_{X/S}^\vee) \cdot \Td(E_{X^t/S}^\vee)  \cdot \ch\bigl(\det(E_{X/S})\bigr) \Bigr) = (-1)^g \cdot [-1]_*(\alpha)\, .
\end{equation}
\end{proposition}

\begin{proof}
Immediate from \eqref{eq:StS}, using the relation between $\cF$ and~$\cS$ that is expressed in~\eqref{eq:SvsFdiag}.
\end{proof}

Deninger and Murre prove that $\cF^t \circ \cF = (-1)^g \cdot [-1]_{X,*}$, see \cite{DenMurre}, Corollary~2.22. Comparison with~\eqref{eq:FtF} gives the following result.

\begin{corollary}\label{maincor}
Let $\cH = \cH^1_\dR(X/S)$ be the de Rham bundle of~$X$ over~$S$. 
\begin{enumerate}
\item\label{maincora} We have $\Td(\cH) = 1$ in $\CH(S;\bbQ)$.

\item\label{maincorb} Suppose $E_{X/S} \cong E_{X^t/S}$, which is the case for instance if there exists a separable isogeny $X \to X^t$. Then $c_i(\cH) = 0$ in $\CH(S;\bbQ)$ for all $i\geq 1$.
\end{enumerate}
\end{corollary}

\begin{proof}
\ref{maincora} The short exact sequence \eqref{eq:sesdR} gives the relation $\Td(\cH) = \Td(E_{X/S})\cdot \Td(E^\vee_{X^t/S})$. On the other hand, $\Td(E) = \ch\bigl(\det(E)\bigr) \cdot \Td(E^\vee)$ for any vector bundle~$E$. Comparing the result of Deninger and Murre with \eqref{eq:FtF}, and taking into account that $\pi^*$ is injective (because $e^*\pi^* = \id$), we obtain that $\Td(\cH) = 1$.

\ref{maincorb} For the total Chern classes, the sequence~\eqref{eq:sesdR} gives the relation $c(\cH) = c(E_{X/S}) \cdot c(E^\vee_{X^t/S})$. If $E_{X/S} \cong E_{X^t/S}$, it follows that all odd Chern classes $c_{2i+1}(\cH)$ vanish. In view of~\ref{maincora}, it now suffices to show that the coefficient of~$c_{2i}$ in the Todd polynomial~$\Td_{2i}$ is nonzero. As noted by Hirzebruch in~\cite{Hirzebruch}, Section~1.7, this coefficient equals the coefficient of~$c_1^{2i}$ in~$\Td_{2i}$. Hence it suffices to show that the coefficient of~$x^{2i}$ in the power series expansion of $x\cdot (1-e^{-x})^{-1}$ is nonzero; but this coefficient equals $(-1)^{i-1}\cdot \frac{B_{2i}}{(2i)!}$, which is nonzero.
\end{proof}

\begin{remark}
Assume there exists a separable isogeny $X\to X^t$. In terms of the Chern classes $\lambda_j = c_j(E_{X/S}) \in \CH^j(S)$, the vanishing of all~$c_i(\cH)$ with rational coefficients means that in $\CH(S;\bbQ)$ we have the identity
\[
\bigl(1+\lambda_1+\cdots + \lambda_g\bigr) \cdot \bigl(1-\lambda_1 + \lambda_2 - \cdots + (-1)^g\lambda_g\bigr) = 1\, .
\]
This result was first proven for principally polarized abelian schemes by van der Geer; see~\cite{vdG}, Theorem~2.1. By Zarhin's trick, this implies the result whenever~$X$ admits a separable polarization.

Much more is known. Let $L$ be a field, let $n\geq 3$ be an integer prime to~$\mathrm{char}(L)$, and let~$\overbar{\mathscr{A}}_{g,n,L}$ be a smooth toroidal compactification of the moduli space~$\cA_{g,n,L}$. (See~\ref{subsec:MRIntro} for the notation.) The de Rham bundle~$\cH$ of the universal abelian scheme over~$\cA_{g,n,L}$ naturally extends to a vector bundle~$\overbar{\cH}$ on~$\overbar{\mathscr{A}}_{g,n,L}$ that comes equipped with a logarithmic flat connection. If $\mathrm{char}(L) = 0$, it was shown by Esnault and Viehweg in~\cite{EV} that $c_i(\overbar{\cH}) = 0$ in $\CH(\overbar{\mathscr{A}}_{g,n,L};\bbQ)$ for all $i\geq 1$. If $\mathrm{char}(L) = p>0$, this same result can be proven using the theory of $F$-zips (with the same argument as in Proposition~\ref{prop:FZip} below). Still in positive characteristic, this has been extended to toroidal compactifications of Shimura varieties of Hodge type by Wedhorn and Ziegler in \cite{WZ}, Theorem~7.1, which is a consequence of the results of Brokemper in~\cite{Brokemper}, Section~2.4. 
\end{remark}

\section{Some results about the orders of the Chern classes of de Rham bundles}

Let $\cA_g$ be the moduli stack over~$\bbZ$ of $g$-dimensional principally polarized abelian varieties (ppav). Let $\cX \to \cA_g$ be the universal abelian scheme, and let $\cH = \cH^1_\dR(\cX/\cA_g)$ be the de Rham bundle. As $\cX$ is isomorphic to its dual, $\cH$ is an extension of $E_{\cX/\cA_g}^\vee$ by~$E_{\cX/\cA_g}$. This implies that $c_{2i+1}(\cH) = 0$ in $\CH(\cA_g)$ for all $i\geq 0$.

In characteristic~$p$, there is an easy method using which we can give a bound on the order of the even Chern classes.

\begin{proposition}\label{prop:FZip}
Let $p$ be a prime number, and let $\cH_{\bbF_p}$ be the de Rham bundle of the universal abelian scheme over~$\cA_{g,\bbF_p}$. Then $(p^{2i}-1)\cdot c_{2i}(\cH_{\bbF_p}) = 0$ in $\CH(\cA_{g,\bbF_p})$, for all $i\geq 1$. 
\end{proposition}

\begin{proof}
For the duration of this proof, let $\cX$ and~$\cA$ mean $\cX_{\bbF_p}$ and~$\cA_{g,\bbF_p}$, and let $\cH = \cH_{\bbF_p}$. The (relative) Frobenius $F_{\cX/\cA} \colon \cX \to \cX^{(p)}$ and Verschiebung $V_{\cX/\cA} \colon \cX^{(p)}\to \cX$ induce homomorphisms $F\colon \cH^{(p)} \to \cH$ and $V \colon \cH\to \cH^{(p)}$, and it is classical that these induce isomorphisms
\[
\cH^{(p)}/\Ker(F) \isomarrow \Ker(V)\, ,\quad \cH/\Ker(V) \isomarrow \Ker(F)\, .
\]
(In fact, $\Ker(F) = E_{\cX/\cA}^{(p)}$.) For the total Chern classes this gives
\[
c(\cH) = c\bigl(\Ker(V)\bigr) \cdot c\bigl(\cH/\Ker(V)\bigr) = c\bigl(\cH^{(p)}/\Ker(F)\bigr) \cdot c\bigl(\Ker(F)\bigr) = c(\cH^{(p)})\, .
\]
As $c_j(\cH^{(p)}) = p^j\cdot c_j(\cH)$, the result follows.
\end{proof}

In Section~4 of~\cite{EkvdG}, Ekedahl and van der Geer determine the order of the even Chern classes $c_{2i}(\cH)$ in cohomology. Before stating their result, we recall a definition.

\begin{definition}
If $k$ is a positive integer, let $n_k$ be the gcd of the numbers $p^{2k}-1$, where $p$ runs over all primes numbers greater than $2k+1$. 
\end{definition}

If $\ell$ is an odd prime number, the $\ell$-adic valuation of~$n_k$ is given by 
\[
v_\ell(n_k) = \begin{cases} 0 & \text{if $\ell-1$ does not divide $2k$,}\\ 1 + v_\ell(k) & \text{if $\ell-1$ divides~$2k$,} \end{cases}
\]
whereas the $2$-adic valuation is given by $v_2(n_k) = 3 + v_2(k)$. The first values are as follows:
\[
\begin{array}{lcccccccccccccc}
k: & 1 & 2 & 3 & 4 & 5 & 6 & 7 & 8 & 9 & 10 & 11 & 12 & 13 & 14\\
n_k: & 24 & 240 & 504 & 480 & 264 & 65520 & 24 & 16320 & 28728 & 13200 & 552 & 131040 & 24 & 6960
\end{array}
\] 
It is known that $n_k$ equals the denominator of $B_{2k}/(-4k)$.

\begin{theorem}[Ekedahl--van der Geer, \cite{EkvdG}]\label{thm:EvdG}
\phantom{x}
\begin{enumerate}
\item\label{EvdGcohom} For $i \in \{1,\ldots,g\}$, the order of $c_{2i}(\cH_\bbC)$ in the singular cohomology group $H^{4i}\bigl(\cA_{g,\bbC},\bbZ(2i)\bigr)$ equals either $n_i$ or~$n_i/2$. Similarly, if $p$ and~$\ell$ are prime numbers with $\ell \neq p$, the order of $c_{2i}(\cH_{\bbF_p})$ in the $\ell$-adic cohomology $H^{4i}\bigl(\cA_{g,\overline\bbF_p},\bbZ_\ell(2i)\bigr)$ equals the $\ell$-part of either $n_i$ or~$n_i/2$.

\item\label{EvdGlambdag} If $L$ is any field, the top Chern class $\lambda_g \in \CH^g(\cA_{g,L})$ of the Hodge bundle $E = E_{\cX_L/\cA_{g,L}}$ has order dividing $(g-1)!\cdot n_g$. 
\end{enumerate}
\end{theorem}

\begin{remarks}\label{rems:EvdG}
\begin{enumerate}
\item A key property of cohomology is that we have the smooth base change theorem, which in the above setting gives isomorphisms $H^{4i}\bigl(\cA_{g,\bbC},\bbZ(2i)\bigr) \otimes \bbZ_\ell \isomarrow H^{4i}\bigl(\cA_{g,\overline\bbF_p},\bbZ_\ell(2i)\bigr)$ ($p\neq \ell$), mapping Chern classes to Chern classes. In particular, the two assertions in part~\ref{EvdGcohom} of the theorem are mutually equivalent.

\item Taking into account the previous remark, it is immediate from Proposition~\ref{prop:FZip} that the order of $c_{2i}(\cH_\bbC)$ in $H^{4i}\bigl(\cA_{g,\bbC},\bbZ(2i)\bigr)$ divides~$n_i$. In~\cite{EkvdG} this is proven in a different way.

\item\label{rem:c2gH} As $c_{2g}(\cH) = (-1)^g \cdot \lambda_g^2$, it follows from part~\ref{EvdGlambdag} of the theorem that also $c_{2g}(\cH) \in \CH^{2g}(\cA_{g,L})$ has order dividing $(g-1)!\cdot n_g$.
\end{enumerate}
\end{remarks}

\subsection{}
Theorem~\ref{thm:EvdG} gives lower bounds for the orders of the even Chern classes $c_{2i}(\cH)$ in the Chow group of $\cA_{g,\bbQ}$ or $\cA_{g,\bbF_p}$. In characteristic~$p$, Proposition~\ref{prop:FZip} gives an upper bound for the order. For large~$p$ this upper bound will be far from sharp, as the actual order of the class in $\CH(\cA_{g,\bbF_p})$ divides the order of $c_{2i}(\cH_\bbQ) \in \CH(\cA_{g,\bbQ})$, which does not depend on~$p$. As we shall discuss, the result can still be very useful when combined with other information.

Let us also observe that the upper bound given by Proposition~\ref{prop:FZip} is consistent with the lower bound given by Ekedahl and van der Geer because for every prime $\ell\neq p$ the $\ell$-part of~$n_i$ divides~$p^{2i}-1$. Note, however, that the order of~$c_{2i}(\cH_{\bbF_p})$ in $\CH(\cA_{g,\bbF_p})$ is not divisible by~$p$, whereas the order of~$c_{2i}(\cH_\bbC)$ in the cohomology of~$\cA_{g,\bbC}$ is divisible by~$p$ whenever $p-1$ divides~$2i$.

\subsection{}\label{subsec:MRIntro}
For the next result, we again need some notation. If $q$ is a nonzero rational number, we write $\num(q)$ for its (absolute) numerator. Further, if $V$ is a vector bundle on some variety~$S$, we define classes $N_i(V) \in \CH^i(S)$ by the rule that $N_i(V)$ is the sum of the $i$th powers of the Chern roots of~$V$.

If $n$ is an integer with $n\geq 3$, let~$\cA_{g,n}$ denote the moduli scheme over~$\bbZ[1/n]$ of $g$-dimensional ppav equipped with a level~$n$ structure. (Whenever we use this notation, we assume $n\geq 3$.) Here we define a level~$n$ structure on a $g$-dimensional ppav $(X,\theta)$ over a scheme~$S$ as a pair $(\alpha,\zeta)$ consisting of isomorphisms of group schemes $\alpha\colon (\bbZ/n\bbZ)_S^{2g} \isomarrow X[n]$ and $\zeta \colon (\bbZ/n\bbZ)_S \isomarrow \mu_{n,S}$, such that the standard symplectic form on $(\bbZ/n\bbZ)_S^{2g}$ agrees with the Weil pairing $e^\theta \colon X[n] \times X[n] \to \mu_n$ via~$\zeta$.

With this notation, we have the following result of Maillot and R\"ossler.

\begin{theorem}[Maillot--R\"ossler, \cite{MailRoss}]\label{thm:MR}
Let $n$ be an even integer with $n\geq 4$, and let $E = E_{\cX/\cA_{g,n,\bbC}}$ be the Hodge bundle of the universal abelian scheme over~$\cA_{g,n,\bbC}$. For all $i\geq 1$, the classes $N_{2i}(E) \in \CH^i(\cA_{g,n,\bbC})$ are torsion. If $\ell$ is a prime number with $\ell > 2i$, we have the estimate
\[
v_\ell\bigl(\text{order of $N_{2i}(E)$}\bigr) \leq v_\ell\bigl(\num((2^{2i}-1) \cdot B_{2i})\bigr)
\]
for the $\ell$-adic valuation of the order of the class~$N_{2i}(E)$.
\end{theorem}

Maillot and R\"ossler in fact prove something stronger, namely the analogue of the above assertion on a smooth toroidal compactification of~$\cA_{g,n,\bbC}$. It is not clear to us if the theorem is also valid over~$\bbQ$ or other base fields; this would be quite useful because by specialization it would then give the corresponding result over all~$\bbF_p$ with $p\nmid n$.

\subsection{}\label{subsec:IAFD}
To some extent the same method as in Section~\ref{sec:FD} can give additional information. This is based on an integral refinement of Fourier duality for abelian schemes that was worked out in~\cite{IAFD}. To state what we need, we  place ourselves in the situation where $S$ is a smooth quasi-projective variety of dimension~$d$ over some field~$L$ and $X \to S$ is an abelian scheme of dimension~$g$. If $\mathrm{char}(L) = p>0$, we assume that there exists a separable isogeny $X \to X^t$ and that $X$ admits a level~$n$ structure for some $n\geq 3$ prime to~$p$. (For instance, $X$ could be the universal abelian scheme over~$\cA_{g,n}$ in any characteristic not dividing~$n$.) Define
\[
\Lambda = \bbZ\Bigl[\frac{1}{(2g+d+1)!}\Bigr]\, .
\]
It is proven in~\cite{IAFD}, Section~5, that all classical results about Fourier duality and Beauville's decomposition remain valid for Chow rings with coefficients in~$\Lambda$. With the same arguments as in Section~\ref{sec:FD}, it follows from this that $\Td\bigl(\cH^1_\dR(X/S)\bigr) = 1$ in $\CH(S;\Lambda)$; see~\cite{IAFD}, Theorem~5.5(1).

\begin{proposition}\label{prop:BUpperbd}
Let notation and assumptions be as in~\ref{subsec:IAFD}, and write $\cH = \cH^1_\dR(X/S)$. For $i\in \{1,\ldots,g\}$, consider the Chern class $c_{2i}(\cH) \in \CH^{2i}(S)$. If $\ell$ is a prime number with $\ell > 2g+d+1$, we have the estimate
\[
v_\ell\bigl(\text{order of $c_{2i}(\cH)$}\bigr) \leq 
\max\{v_\ell(B_2),\ldots,v_\ell(B_{2\lfloor\frac{i}{2}\rfloor}),v_\ell(B_{2i})\}\, .
\]
\end{proposition}

\begin{proof}
The degree~$2i$ component of the Todd polynomial is of the form
\[
\Td_{2i} = f(c_1,\ldots,c_{2i-1}) + (q\cdot B_{2i})\cdot c_{2i}\, ,
\]
where $f$ is a homogeneous polynomial of degree~$2i$ in $\Lambda[c_1,\ldots,c_{2i-1}]$ (with $\mathrm{deg}(c_j) = j$) and $q\in \Lambda^*$. Using induction on~$i$ (with $i\in \{1,\ldots,g\}$), the assertion now follows from the fact that $\Td(\cH) = 1$ in $\CH(S;\Lambda)$ together with the vanishing of the odd Chern classes. (Note that each monomial in~$f$ involves at least one class~$c_j$ with $j \leq i$.)
\end{proof}

If $\mathrm{char}(L) = p>0$, we can combine this with Proposition~\ref{prop:FZip}, and we find that the $\ell$-adic valuation of the order of $c_{2i}(\cH)$ is at most the minimum of $v_\ell(p^{2i}-1)$ and $\max\{v_\ell(B_2),\ldots,v_\ell(B_{2\lfloor\frac{i}{2}\rfloor}),v_\ell(B_{2i})\}$.

If we try to apply these results, we are confronted with the fact that the even Bernoulli numbers can have large numerators; see~\cite{OEIS}, sequence A027641. It seems a general pattern that these numerators have large prime factors.

\begin{example}
Let $X$ be the universal abelian scheme over $S = \cA_{13,4}$, the moduli scheme of $13$-dimensional ppav with a level~$4$ structure. In this case $g=13$ and $d = 91$, so that $\Lambda = \bbZ\bigl[1/105!\bigr]$. Using the relation $\Td(\cH) = 1$ in $\CH(S;\Lambda)$, we find that $c_2(\cH),\ldots,c_{10}(\cH)$ are $\ell$-adic units for all prime numbers $\ell > 105$. As $\num(B_{12}) = 691$, we find that $c_{12}(\cH)$ satisfies $691\cdot c_{12}(\cH) = 0$ in~$\CH(S;\Lambda)$, and $691$ is prime. Continuing in this manner we find in $\CH(S;\Lambda)$ the relations
\begin{itemize}
\item $3617 \cdot c_{14}(\cH) = 0$, and $3617$ is prime;
\item $43867 \cdot c_{16}(\cH) = 0$, and $43867$ is prime;
\item $174611\cdot c_{18}(\cH) = 0$, and $174611 = 283 \cdot 617$;
\item $854513 \cdot c_{20}(\cH) = 0$, and $854513 = 11 \cdot 131 \cdot 593$;
\item $236364091 \cdot c_{22}(\cH) = 0$, and $236364091 = 103 \cdot 2294797$.
\end{itemize}
For the next even Bernoulli number we have $\num(B_{24}) = 8553103$, and $8553103 = 13 \cdot 657931$. However, it may happen that in the Todd polynomial~$\Td_{24}$ the coefficient of~$c_{12}^2$ is nonzero. (We do not know if this is indeed the case, as the Todd polynomials are rather hard to compute.) If this happens, the order of~$c_{24}(\cH)$ could be divisible by~$691$. 

For the next one, we have $\num(B_{26}) = 23749461029 = 7 \cdot 9349 \cdot 362903$; so the order of $c_{26}(\cH)$ could, a priori, be divisible by one or more of the prime factors $691$, $3617$, $9349$ or~$362903$. However, this does not happen because $c_{26}$ is the top Chern class, so the result of Ekedahl and van der Geer (see Remark~\ref{rems:EvdG}\ref{rem:c2gH}) tells us that the order of $c_{26}(\cH)$ in~$\CH(S)$ divides $12!\cdot n_{12} = 2^5 \cdot 3^2\cdot 5 \cdot 7 \cdot 13!$.

If, still with the same $X/S$, the base field has characteristic $p>0$, we can often do better. For instance, we then know that $(p^{24}-1) \cdot c_{12}(\cH) = 0$, but from the above we know that $691 \cdot c_{12}(\cH) = 0$ in~$\CH(S;\Lambda)$. As $\mathrm{gcd}(24,690) = 6$, it follows that the order of~$c_{12}(\cH)$ only has prime factors smaller than~$105$, unless~$p$ satisfies $p^6\equiv 1 \bmod{691}$, in which case also $691$ may divide the order of~$c_{12}(\cH)$.
\end{example}

\subsection{}
Finally, we may try to compare Proposition~\ref{prop:BUpperbd} with what we get from Theorem~\ref{thm:MR}. Consider a situation as in~\ref{subsec:IAFD}, with base field $L=\bbC$. The even Chern classes of $\cH = \cH^1_\dR(X/S)$ can be expressed as polynomials with coefficients in~$\Lambda$ in the even Newton classes $N_{2j}(E)$ of $E = E_{X/S}$, using the identities $N_k = (-1)^{k-1}\cdot \lambda_k +  \sum_{i=1}^{k-1}\; (-1)^{k-1+i} \cdot \lambda_{k-i}\cdot N_i$. (Recall that $\lambda_j = c_j(E)$.) E.g.,
\begin{align*}
c_2(\cH) &= 2 \lambda_2 - \lambda_1^2 = - N_2(E)\\
c_4(\cH) &= 2 \lambda_4 - 2 \lambda_1 \lambda_3 + \lambda_2^2 = \tfrac{1}{2} \cdot \bigl(N_2(E)^2 - N_4(E)\bigr)\\
c_6(\cH) &= 2 \lambda_6 - 2 \lambda_1 \lambda_5 + 2 \lambda_2 \lambda_4 - \lambda_3^2 = -\tfrac{1}{6} \cdot \bigl(N_2(E)^3 - 3 N_2(E) N_4(E) + 2 N_6(E)\bigr)\\
c_8(\cH) &= 2 \lambda_8 - 2 \lambda_1 \lambda_7 + 2 \lambda_2 \lambda_6 - 2 \lambda_3 \lambda_5 + \lambda_4^2\\
&= \tfrac{1}{24} \bigl(N_2(E)^4 - 6 N_2(E)^2 N_4(E) + 8 N_2(E) N_6(E) +3 N_4(E)^2 - 6 N_8(E)\bigr)\, .
\end{align*}
Using such relations, it seems that, over~$\bbC$, the bounds that we get from Proposition~\ref{prop:BUpperbd} are essentially the same as those obtained from Theorem~\ref{thm:MR}, though the latter also gives information about primes~$\ell$ with $2g< \ell \leq 2g+d+1$. (On the other hand, Proposition~\ref{prop:BUpperbd} has the advantage that it applies to abelian schemes over any base scheme that is smooth and quasi-projective over a field.)

%%
%% make bibliography small:
%%
%% end of small for the blibliography
\bigskip

\noindent
\texttt{b.moonen@science.ru.nl}

\noindent
Radboud University Nijmegen, IMAPP, Nijmegen, The Netherlands

\end{document}